\documentclass[english,10pt,oneside]{amsart}

\usepackage[utf8]{inputenc}
\usepackage[T1]{fontenc}
\usepackage{frcursive}
\usepackage{lmodern}
\usepackage[a4paper]{geometry}
\usepackage[english]{babel}
\usepackage{xcolor}
\usepackage{hyperref}

\usepackage{csquotes}
\usepackage[backend=biber]{biblatex}
\usepackage{listings}

\usepackage{amsmath,mathtools,amssymb,mathrsfs,eucal}

\usepackage{graphicx}
\usepackage{tikz-cd}
\usepackage[all]{xy}

\usepackage{amsthm}

\usepackage{todonotes}

\usetikzlibrary{matrix,arrows}

\swapnumbers

\definecolor{darkgreen}{rgb}{0,0.45,0}

\hypersetup{
colorlinks,
linktoc=all,
linkcolor={darkgreen},
citecolor={blue!50!black},
urlcolor={blue!80!black}
}

\theoremstyle{definition}

\theoremstyle{plain}
\newtheorem{theo}{Théorème}[section]

\newtheorem{theoa}{Theorem}[section]

\newtheorem{prop}[theo]{Proposition}

\newtheorem{coroa}[theo]{Corollary}

\newtheorem{defia}[theo]{Definition}

\newtheorem{lemma}[theo]{Lemma}

\theoremstyle{remark}

\newtheorem{remaa}[theo]{Remark}

\newtheorem{exema}[theo]{Example}

\numberwithin{equation}{subsubsection}

\usepackage{stmaryrd}

\numberwithin{equation}{subsubsection}
\author{Mathieu Daylies}
\address{mathieu.daylies@mathematik.tu-chemnitz.de }
\title{Base Change for coherent cohomology in Berkovich geometry}

\usepackage{stmaryrd}

\def\Spec{\operatorname{Spec}}

\theoremstyle{plain}
\newtheorem{theointro}{Theorem}

\begin{document}


\maketitle

We prove some base change theorems for coherent cohomology in the setting of Berkovich spaces. In this setting, we get a flat base change theorem, and some proper base change theorems that are analogue to theorems from scheme theory.

\section{Introduction}

\subsection{Motivation}

  When considering a cartesian square of spaces, base change theorems relate how the direct and the inverse image of sheaves fit together. To be more precise, let $$ \xymatrix{
    X' \ar[r]^{f'} \ar[d]^{p'}  & S' \ar[d]^{p} \\
    X \ar[r]^{f} & S
  }$$
 be a cartesian square in some category of spaces with sheaves on it (schemes, topological spaces, analytic spaces,...) and let $\mathcal{F}$ be a sheaf on $X$. Then base change theorems are statements concerning the natural transformation of sheaves $G:p^*R^i(f_*\mathcal{F})\to R^if'_* (p'^* \mathcal{F})$. 
 
 Suppose that all the spaces involved in the previous diagram are schemes, and $\mathcal{F}$ is a quasi-coherent sheaf. Then the flat base change theorem from algebraic geometry states that if $p$ is flat, and $f$ is quasi-compact and quasi-separated, then $G$ is an isomorphism (see  \cite[\href{https://stacks.math.columbia.edu/tag/02KH}{Tag 02KH}]{stacks-project}). If the base $S$ is noetherian, $f$ is proper, and $\mathcal{F}$ is flat over $S$, we also have proper base change theorems for coherent cohomology that are a bit weaker (see  \cite[\href{https://stacks.math.columbia.edu/tag/07VL}{Tag 07VL}]{stacks-project} or chapter 2, Section 5 of \autocite{mumford1974abelian}).
 
 Proper base change theorems are especially useful when we try to figure out how cohomology behaves in familly : we can often relate the cohomology on the fiber and  the cohomology on some neighboorhood of a point (cf corollary 2 and 3 of section 5 of \autocite{mumford1974abelian}). With motivation linked to the study of relative ampleness locus of a morphism between Berkovich analytic spaces, we are interested in those base change theorems for coherent cohomology in the non-archimedean analytic setting. Berkovich stated some of these theorems at the end of section 3.3 of \autocite{berkovich2012spectral}. In his work, these theorems are deduced from theorem 3.3.9 of loc.cit., whose proof is supposed to be completely analogous to the proofs of the corresponding facts in algebraic geometry. However, in analytic geometry, the tensor products that appear are replaced by completed tensor products which are not really compatible with tools from homological algebra used in the previous proof. This fact leads us to think that the proof in loc. cit. is not correct.
 
 Note that in Berkovich geometry, we also have an étale cohomology and some base change theorems for étale cohomology (see \autocite{https://doi.org/10.48550/arxiv.1101.0683} for a reader friendly introduction to étale cohomology of Berkovich analytic spaces, and section 7.7 and 7.8 of \autocite{PMIHES_1993__78__5_0} for the precise theorems), but we only deal with coherent cohomology here.
 
 \subsection{Overview of our results}
 
 \subsubsection{Flat base change for coherent cohomology}
 
In Berkovich geometry, we have at our disposal a flatness theory, due to Ducros in \autocite{ducros2017families}. In algebraic geometry, the easiest base change theorem is the flat base change theorem, so it is natural to try to prove such a theorem in the analytic setting. We obtain the following theorem $\ref{bigtheoflat}$: 

\begin{theointro} \label{flatintro}Suppose $$ \xymatrix{
    X' \ar[r]^{f'} \ar[d]^{p'}  & S' \ar[d]^{p} \\
    X \ar[r]^{f} & S
  }$$
is a cartesian diagram of $k$-analytic spaces with $f$ proper and $p$ flat.
Let $\mathcal{F}$ be a coherent sheaf on $X$. Then for all $i\geq 0$ the natural morphism of coherent $\mathcal{O}_{S'}$-modules $p^*(R^i f_* \mathcal{F})\to R^i f'_* (p'^* \mathcal{F})$ is an isomorphism.
\end{theointro}

We can reduce this statement to the case where $S=\mathcal{M(A)}$ and $S'=\mathcal{M(B)}$ are both affinoid, and then it becomes a theorem on the Čech complex $C^*$ of $\mathcal{F}$ associated to some affinoid covering of $X$. The Čech complex of $X_\mathcal{B}$ is now $C^*\hat{\otimes}_\mathcal{A}\mathcal{B}$ and we just need to show that the natural map $H^i(C^*)\otimes_\mathcal{A} \mathcal{B}\to H^i(C^*\hat{\otimes}_\mathcal{A}\mathcal{B})$ is an isomorphism. In scheme theory, this would be straightforward because tensor product by a flat algebra commutes by definition with cohomology of a complex. On the contrary, in analytic geometry, even if $p:S'=\mathcal{M(B)}\to S=\mathcal{M(A)}$ is flat, the functor $- \hat{\otimes}_\mathcal{A}\mathcal{B}$ needs not to be exact in any sense.

We introduce in $\ref{defipa}$ the property $\mathcal{P_A}$. An $\mathcal{A}$ Banach module $M$ satisfies the property $\mathcal{P_A}$ if the natural arrow $H^i(C^*)\otimes_\mathcal{A} \mathcal{B}\to H^i(C^*\hat{\otimes}_\mathcal{A}\mathcal{B})$ is an isomorphism. In order to show the theorem, it is sufficient to show that $\mathcal{B}$ satisfies the property $\mathcal{P_A}$. 

The first step is to show, by
direct computation involving explicit polydisks, that any quasi-smooth algebra overs $\mathcal{A}$ satisfies $\mathcal{P_A}$. We then show that the $\mathcal{P_A}$ property behaves well with respect to composition, and that this property is local for the $G$-topology on $S$. The last and crucial step is now to use the analytic version of Raynaud-Gruson's theory of dévissages, due to Ducros in section 8.2 of \autocite{ducros2017families}, to handle the general case of an affinoid algebra $\mathcal{B}$, analytically flat over $\mathcal{A}$.

\subsubsection{Proper base change for coherent cohomology}

In 3.3.9 of \autocite{berkovich2012spectral}, Berkovich claimed that we could approach these problems in the non-archimedean analytic setting with the strategy used by Mumford in chapter 5 of \autocite{mumford1974abelian} for schemes. The idea is to show that if we have a proper scheme $X\to \Spec A$ over $A$, $\mathcal{F}$ a coherent sheaf over $X$, flat over $A$, then there exists a complex of \textit{finite} and flat $A$-modules $K^*$ that compute \textit{universally} the cohomology of $\mathcal{F}$ in the following sense: for all $A$-algebra $B$, we have an isomorphism $H^i(K^*\otimes_A B)\to H^i(X_B,\mathcal{F}_B)$. To prove this, we use the Čech complex of $\mathcal{F}$ relative to some open cover, and we use tools from homological algebra. The finiteness of the $K^i$'s is the crucial point of this theorem, because the Čech complex does not usually contain any finite module. Classic proper base change theorems are then a consequence of the existence of this complex.

If we try to adapt this proof to the non-archimedean analytic setting, it seems to fail. Indeed, if we consider a proper analytic space $X\to \mathcal{A} $ over an affinoid algebra $\mathcal{A}$, $\mathcal{F}$ a coherent sheaf over $X$, flat over $\mathcal{A}$, the Čech complex of $\mathcal{F}$ does not consist of type module over $\mathcal{A}$, so if $p:\mathcal{M(B)}\to\mathcal{M(A)}$ is a morphism of affinoid algebras, the tensor products appearing in the Čech complex of $\mathcal{F}_\mathcal{B}$ are \textit{completed} tensor products that prevent us from using the tools from homological algebra.

We however get the following theorem $\ref{propBC}$: 

\begin{theointro}\label{BCintro} Let $\mathcal{A}$ be a $k$-affinoid algebra, $X$ a proper $\mathcal{A}$-analytic space, and $\mathcal{F}$ a coherent sheaf on $X$, that is flat over $\mathcal{M(A)}$. Then there exists a finite complex $K^*$ of finite and projective $\mathcal{A}$-modules such that for any non-archimedean complete field extension $L$ of $k$, and any $L$-affinoid algebra $\mathcal{B}$ with a morphism of analytic spaces $p:\mathcal{M(B)}\to \mathcal{M(A)}$, we have, for all $n\in \mathbb{N}$, an isomorphism $H^n(K^*\otimes_\mathcal{A}\mathcal{B})\to H^n(X_\mathcal{B},\mathcal{F}_\mathcal{B})$.

\end{theointro}

The idea of the proof is to work with the Čech complex of $\mathcal{F}$, and to note that the theorem is true in the following cases:

\begin{enumerate}
    \item if $\mathcal{B}=\mathcal{A}_L$ with $L$ a non-archimedean complete extension of $k$ because $\hat{\otimes}_k L$ transforms exact admissible sequences into exact admissible sequences by Gruson's theory in \autocite{gruson1966theorie},
    \item if $\mathcal{B}$ is $k$-affinoid, and $p$ is analytically flat by theorem $\ref{flatintro}$, 
    \item if $p$ is a finite morphism between $k$-affinoid spaces by the algebraic theory and the construction by Mumford in Section 5 of \autocite{mumford1974abelian}, because in this case, \textit{completed} tensor product and usual tensor product are the same.  
\end{enumerate}

Then, we can use Ducros's work on the relative dimension in \autocite{ducros_2007}, and in particular corollary 4.7 which gives a nice $G$-local factorization of some morphisms and allows to conclude.

Theorem $\ref{BCintro}$ allows us to derived classical versions of proper base change theorems in Berkovich geometry. In particular, we obtain the upper semi-continuity of the dimension of the cohomology group of the fiber for the Zariski topology, and the fact that the Euler characteristic is locally constant.

We also find back a classic statement as a corollary in theorem $\ref{finver}$:

\begin{theointro}Let $f:X\to S$ be a proper morphism of $k$-analytic spaces with $S$ connected and reduced, and $\mathcal{F}$ a coherent sheaf on $X$ that is flat on $S$. Then, for all $p\in \mathbb{N}$, there is an equivalence between:
\begin{enumerate}
    \item the function $s\in S\mapsto \dim_{\mathcal{H}(s)} H^p(X_s,\mathcal{F}_s)$ is constant ;
    \item the sheaf $R^p f_*(\mathcal{F})$ is locally free on $S$, and for all $s\in S$, the natural map $R^p f_*(\mathcal{F})\otimes_{\mathcal{O}_S} \mathcal{H}(s)\to H^p(X_s,\mathcal{F}_s)$ is an isomorphism.
\end{enumerate}
If these equivalent conditions are fulfilled, then the natural map $R^{p-1} f_*(\mathcal{F})\otimes_{\mathcal{O}_S} \mathcal{H}(s)\to H^{p-1}(X_s,\mathcal{F}_s)$ is an isomorphism for all $s\in S$.

\end{theointro}

\subsection{Acknowledgments} I would like to express my sincere gratitude to Antoine Ducros, without whom this work would not have been possible. I would also like to thank Christian Lehn for his welcome in Chemnitz University.

This article was completed when the author was supported through the SMWK research grant SAXAG.


\section{Flat base change}

The main goal of this section is to show a flat base change theorem for coherent modules in Berkovich geometry. Our strategy will be to use Ducros's theory of dévissages in Berkovich geometry (see chapter 8 of \autocite{ducros2017families}) to reduce the theorem to two elementary cases, the case of a quasi-smooth base change, and the case of a $G$-covering. The case of a $G$-covering follows from Kiehl's theorem on direct images, and the case of quasi-smooth morphisms will be handled by direct computation involving explicit polydisks.

We first give a reminder of Ducros theory of dévissages and give the explicit version that we will use here. The following theorem is just a version of the theorem 8.3.4 of \autocite{ducros2017families} where we shrinked every quasi-smooth space.

\begin{prop}[Ducros] \label{deviss}Let $p:Y\to X$ be a morphism between good $k$-analytic spaces, let $\mathcal{F}$ be a flat coherent sheaf on $Y$, let $y$ be a point of $\mathrm{supp}~Y$, let $x$ be its image in $X$ and let $n$ be the relative dimension of $p$ at $x$. Then there exist $r\leq n$, a decreasing sequence of non-negative integers $n_1>n_2>...>n_r$ and a list of data $(V,\{T_i,\pi_i,t_i,L_i, P_i\}_{i\in \{1,...r\}})$, where : 
\begin{enumerate}
    \item $V$ is an affinoid neighboorhood of $y\in Y$;
    \item $T_i=\mathcal{M}(C_i)$ is a $k$-affinoid domain of a smooth $X$-space of pure relative dimension $n_i$ and $t_i$ is a point of $T_i$ lying over $x$;
    \item for every $i$, $L_i$ and $P_i$ are finite $C_i$-modules, and $L_i$ is free over $C_i$;
    \item $t_i\in \mathrm{supp}~(P_i)$ if $i<r$ and $P_r=0$;
    \item $\pi_1$ is a finite $X$-map from $\mathrm{supp}~ (\mathcal{F}_V)$ to $T_1$ such that $\pi_1^{-1}(t_1)=\{y\}$ set-theoretically;
     \item $\pi_i$ is a finite $X$-map from $\mathrm{supp}~ (P_{i-1})$ to $T_i$ such that $\pi_1^{-1}(t_i)=\{t_{i-1}\}$ set-theoretically;
     \item There exist an exact sequence of finite $C_1$-modules $0\to L_1\to H^0(V,\mathcal{F})\to P_1\to 0$ ; 
     \item for any $i\in \{2,...,r\}$, there is an exact sequence of finite $C_i$-modules $0\to L_i\to P_{i-1}\to P_i\to 0$.
\end{enumerate}
\end{prop}

\begin{proof}[Proof]

Let $(V',\{T_i,\pi_i,t_i,\mathcal{L}_i,\mathcal{P}_i\}_{i\in \{1,...,r\}})$ be an $X$-dévissage of $\mathcal{F}$ at $y$ (such a dévissage exist by theorem 8.2.3 of \autocite{ducros2017families}). Then the arrow $\mathcal{L}_1\to {\pi_1}_* \mathcal{F}_V$ is injective at $t_1$ and the arrow $\mathcal{L}_i\to {\pi_i}_* \mathcal{P}_{i-1}$ is injective at $t_i$ for every $i\geq 2$ by theorem 8.3.4 of \autocite{ducros2017families} and flatness of $\mathcal{F}$.

By definition, and by theorem 2.4.9 of \autocite{ducros2017families}, we can shrink $V'$ and all the $T_i$ to reduce to the case where if $L_i$ (resp. $P_i$) is the $O_{T_i}(T_i)$-module associated with $\mathcal{L}_i$ (resp. $\mathcal{P}_i$) then the arrows $L_i\to P_{i-1}$ and $L_1\to H^0(V,\mathcal{F})$ are injective and the sequences $0\to L_i\to P_{i-1}\to P_i\to 0$ for all $i>1$ and $0\to L_1\to H^0(V,\mathcal{F})\to P_1\to 0$ are exact. \qedhere

\end{proof}

\begin{remaa}
By the reverse implication in 8.3.4 of \autocite{ducros2017families}, the coherent module associated to $P_i$ over $T_i$ is flat over $\mathcal{M(A)}$ because we have a dévissage of it $(V,\{T_i,\pi_i,t_i,L_i, P_i\}_{i\in \{2,...r\}})$ that is just the restriction of the dévissage of $\mathcal{F}$ 
\end{remaa}


The main theorem of this section is the following: 

\begin{theoa} \label{bigtheoflat}Suppose $$ \xymatrix{
    X' \ar[r]^{f'} \ar[d]^{p'}  & S' \ar[d]^{p} \\
    X \ar[r]^{f} & S
  }$$
is a cartesian diagram of $k$-analytic spaces with $f$ proper and $p$ flat.
Let $\mathcal{F}$ be a coherent sheaf on $X$. Then for all $i\geq 0$ the natural morphism of coherent $\mathcal{O}_{S'}$-modules $p^*(R^i f_* \mathcal{F})\to R^i f'_* (p'^* \mathcal{F})$ is an isomorphism.
\end{theoa}

The rest of this section is devoted to the proof of this theorem.

\begin{remaa} \label{affbc}
By Kiehl theorem on coherence of direct images, theorem 3.3 in \autocite{kiehl1967endlichkeitssatz}, the theorem holds for every embedding of analytic domain $p:S'\to S$.
\end{remaa}

Since this assertion is local on $S'$ for the $G$-topology, we can assume than $S'$ is affinoid, and the theorem is just a consequence of the following theorem: 

\begin{theoa} \label{aff}Suppose $$ \xymatrix{
    X' \ar[r]^{f'} \ar[d]^{p'}  & S' \ar[d]^{p} \\
    X \ar[r]^{f} & S
  }$$

is a cartesian diagram of $k$-analytic spaces with $f$ proper and $p:S':=\mathcal{M(B)}\to S:=\mathcal{M(A)}$ a flat morphism between affinoid spaces.
Let $\mathcal{F}$ be a coherent sheaf on $X$. Then for all $i\geq 0$ the natural morphism of $B$-modules $H^i(X,\mathcal{F})\otimes_\mathcal{A} \mathcal{B}\to H^i(X_\mathcal{B},\mathcal{F}_B)$ is an isomorphism.

\end{theoa}

\begin{remaa} The space $X$ is compact so if we take a $G$-covering of $X$ by a finite number of affinoid domains, the $i$-th cohomology module $ H^i(X,\mathcal{F})$ is given by $H^i(C^*)$ the $i$-th cohomology module of the Čech complex associated to $\mathcal{F}$, where the latter is denoted by $(C^*)$.  
\end{remaa}

\begin{defia}\label{defipa}Let $\mathcal{A}$ be a $k$-affinoid algebra, and $M$ be a Banach module over $\mathcal{A}$. We say that $M$ satisfies the property $\mathcal{P}_\mathcal{A}$ if for every $i\geq 0$ and all coherent sheaf $\mathcal{F}$ on a proper $\mathcal{A}$-space $X$ provided with a finite $G$-covering the natural arrow $H^i(C^*)\hat{\otimes}_{\mathcal{A}} M\to H^i(C^*\hat{\otimes}_\mathcal{A} M)$ is an isomorphism where $(C^*)$ is the Čech complex of $\mathcal{F}$.
\end{defia}

\begin{remaa}
By properness, the $H^i(C^*)$ are finite $\mathcal{A}$-modules, so we have the equality $H^i(C^*)\hat{\otimes}_{\mathcal{A}} M=H^i(C^*){\otimes}_{\mathcal{A}} M$.
\end{remaa}

\begin{defia} In the setting of theorem $\ref{bigtheoflat}$, we say that a morphism $p:S'\to S$ satisfies property $\mathcal{Q}_S$ if for all proper morphism $f:X\to S$ and $\mathcal{F}$ coherent sheaf on $X$, the conclusion of theorem $\ref{bigtheoflat}$ are satisfied.

\end{defia}
\begin{remaa} For a morphism of $k$-affinoid spaces $p:S'\to S$ the following properties are equivalent : $\mathcal{O}(S')$ satisfies $\mathcal{P}_{\mathcal{O}(S)}$ and the morphism $p$ satisfies $\mathcal{Q}_S$. Nevertheless, the property $\mathcal{Q}$ makes sense for non-affinoid analytic spaces.
\end{remaa}
\begin{exema}
Let $\mathcal{M(A)}=\bigcup_{j\in J} \mathcal{M}(\mathcal{A}_j)$ be a covering of an affinoid space by a finite number of affinoid domains. Let $M:=\prod_{j\in J} \mathcal{A}_j $. Then as in $\ref{affbc}$, by Kiehl coherence theorem, $M$ sastisfies the property $\mathcal{P_A}$. 
\end{exema}

\begin{remaa}In the setting of theorem $\ref{aff}$, let $X=\bigcup_{j\in J} X_j$ be a covering of $X$ by a finite number of affinoid domains. Then we have a finite $G$-covering of $X_\mathcal{B}$ by affinoid domains, namely $X_\mathcal{B}=\bigcup_{i\in I} X_i\times_\mathcal{A} \mathcal{B}$, and by properness we can compute its cohomology groups using Čech cohomology, so we have the equality $H^i(X_\mathcal{B},\mathcal{F}_B)=H^i(C_\mathcal{B}^*)$, where $C_\mathcal{B}^i=(C^i)\hat{\otimes}_\mathcal{A} \mathcal{B}$. 
\end{remaa}

In order to show theorem $\ref{aff}$, it is sufficient to show that $H^i(C_\mathcal{B}^*)=H^i(C)\otimes_\mathcal{A} \mathcal{B}$, that is to say that the Banach $\mathcal{A}$-module $\mathcal{B}$ satisfies the property $\mathcal{P}_\mathcal{A}$. This will be our main goal in the rest of this section. 

The following lemma shows in particular that the property $\mathcal{P}$ is $G$-local on the source.

\begin{lemma} \label{redaf}
Let $p:\mathcal{M(B)}\to \mathcal{M(A)}$ be a morphism of $k$-affinoid spaces, and let $q:\mathcal{M(C)\to \mathcal{M(B)}}$ be a morphism of $k$-affinoid spaces.
\begin{enumerate}
    \item Assume that the induced $k$-algebra morphism $\mathcal{B}\to \mathcal{C}$ is faithfully flat and sastisfies $\mathcal{P}_\mathcal{B}$ (e.g. the morphism $q$ is a finite affinoid $G$-covering of $\mathcal{M(B)}$). Then if $\mathcal{C}$ satisfies the property $\mathcal{P}_\mathcal{A}$ then $\mathcal{B}$ satisfies the property $\mathcal{P}_\mathcal{A}$.
    \item Assume that $p$ satisfies $\mathcal{P_A}$ and $q$ satisfies $\mathcal{P}_B$. Then $p\circ q$ satisfies $\mathcal{P_A}$.
    \item Suppose that there exist a $G$-covering $\mathcal{M(B)}=\bigcup_{i\in I} \mathcal{M}(\mathcal{B}_{V_i})$ of $\mathcal{M((B)}$ such that for all $i\in I$, the $\mathcal{A}$-module $\mathcal{B}_{V_i}$ satisfies the property $\mathcal{P_A}$. Then $\mathcal{B}$ satisfies the property $\mathcal{P_A}$. 
\end{enumerate}

\end{lemma}

\begin{proof}[Proof]
Let $X$ be a proper $\mathcal{A}$-space provided with a covering by a finite number of affinoid domains. Denote by $(C^*)$ the associated Čech complex. 

We start with the first point. For all $i\geq 0$, we have a natural arrow $H^i(C^*)\otimes_\mathcal{A}\mathcal{B}\to H^i(C^*\hat{\otimes}_\mathcal{A} \mathcal{B})$. By faithfull flatness, this arrow is an isomorphism if and only if its base change $H^i(C^*)\otimes_\mathcal{A}\mathcal{B}\otimes_\mathcal{B} \mathcal{C}\to H^i(C^*\hat{\otimes}_\mathcal{A} \mathcal{B})\otimes_\mathcal{B} \mathcal{C}$ is an isomorphism, but because $\mathcal{C}$ satisfies $\mathcal{P_A}$, the first term is isomorphic to  $H^i(C^*\hat{\otimes}_\mathcal{A} \mathcal{C})$ and because $q$ satisfies $\mathcal{P}_\mathcal{B}$, we have an isomorphism $H^i(C^*\hat{\otimes}_\mathcal{A} \mathcal{B})\otimes_\mathcal{B} \mathcal{C}\to H^i(C^*\hat{\otimes}_\mathcal{A} \mathcal{B}\hat{\otimes}_\mathcal{B}\mathcal{C})=H^i(C^*\hat{\otimes}_\mathcal{A} \mathcal{C})$.

For the second point, we have an isomorphism $H^i(C^*)\otimes_\mathcal{A}\mathcal{C}= H^i(C^*)\otimes_\mathcal{A} \mathcal{B}\otimes_\mathcal{B}\mathcal{C}\to H^i(C^*\hat{\otimes}_\mathcal{A} \mathcal{B})\otimes_\mathcal{B} \mathcal{C}$ because $p$ satisfies $\mathcal{P_A}$ and this last term is isomorphic to $H^i(C^*\hat{\otimes}_\mathcal{A} \mathcal{B}\hat{\otimes}_\mathcal{B} \mathcal{C})$ because $q$ satisfies $\mathcal{P_B}$.

The last point is just a consequence of the first point, because by compactness, we can chose $I$ to be finite, and then, the natural map $\coprod_{i\in I} \mathcal{M}(\mathcal{B}_{V_i})\to \mathcal{A}$ satisfies $\mathcal{P_A}$ and the arrow $\mathcal{B}\to \prod_{i\in I} \mathcal{B}_{V_i}$ is faithfully flat.
\end{proof}

We have the same kind of good behaviour for the property $\mathcal{Q}$. Let's first introduce the following definition:

\begin{defia} Let $f:Y\to X$ be a morphism of $k$-analytic spaces. Then $f$ is said to be properly surjective if there exist a $G$-covering of $X$ by quasi-compact analytic domain each of which is the image of a quasi-compact analytic domain of $Y$.
\end{defia}

\begin{lemma} \label{redQ}
Let $p:S\to T$ and $q:R\to S$ be morphism of $k$-analytic spaces. 
\begin{enumerate}
\item Assume that $q$ is flat and properly surjective and satisfies property $\mathcal{Q}_\mathcal{S}$ (e.g. the morphism $q$ is a $G$-covering). Then if $p\circ q$ satisfies $\mathcal{Q}_T$ then $p$ satisfies $\mathcal{Q}_T$.  
\item Assume that $p$ satisfies $\mathcal{Q}_T$ and $q$ satisfies $\mathcal{Q}_S$. Then $p\circ q $ satisfies $\mathcal{Q}_T$.
\item Assume that there exist a $G$-covering $S=\bigcup_{i\in I} S_i$ by quasi-compact analytic domains such that for all $i\in I$, the arrow $S_i\to T$ satisfies $\mathcal{Q}_T$. Then $p$ satisfies $\mathcal{Q}_T$.
\end{enumerate}
\end{lemma}

\begin{proof}[Proof]
Let $f:X\to T$ be a proper $T$-space and $\mathcal{F}$ be a coherent sheaf on $X$. Let $X_S:=X\times_T S$ and $X_R:=X\times_T T$ and denote by $p':X_S\to X$, $q':X_R\to X_S$, $f_S:X_S\to S$ and $f_R:X_R\to R$ be the canonical morphisms. We have the following double cartesian commutative square:

$$ \xymatrix{
    R \ar[r]^{q}   & S \ar[r]^{p} & T \\
    X_R  \ar[u]^{f_R} \ar[r]^{q'} & X_S \ar[u]^{f_S} \ar[r]^{p'} & X\ar[u]^f
  }$$

Assume that $p\circ q$ satisfies $\mathcal{Q}_T$ and $q$ is flat and properly surjective and satisfies $\mathcal{Q}_S$, and let $h:p^* (R^i f_*\mathcal{F})\to (R^i (f_S)_* (p')^*\mathcal{F})$ be the canonical morphism. Then we can pullback $h$ by $q^*$ and because $q$ satisfies $\mathcal{Q}_S$, $q^*h$ is the canonical morphism $(p\circ q)^*(R^i f_* \mathcal{F})\to (R^i (f_R)_* (p'\circ q')^*\mathcal F)$ and the latter is an isomorphism because $p\circ q $ satisfies $\mathcal{Q}_T$. By descent of flat properly surjective morphisms (theorem 3.13 of \autocite{daylies2021descente}), $h$ is an isomorphism, so $p$ satisfies $\mathcal{Q_T}$. 

Assume now that $p$ satisfies $\mathcal{Q}_T$ and $q$ satisfies $\mathcal{Q}_S$. Then we have an isomorphism $h:p^* R^i f_* \mathcal{F}\to R^i((f_S)_*(p')^*\mathcal{F})$. We can now pullback this isomorphism by $q^*$ to get another isomorphism $q^*h$ and because $q$ satisfies $\mathcal{Q}_S$, $q^* h$ is the canonical morphism $(p\circ q)^*(R^i f_* \mathcal{F})\to (R^i (f_R)_* ((p'\circ q'))^*\mathcal F)$, so $p\circ q$ satisfies $\mathcal{Q}_T$.

Now, the last point is just a consequence of the first point, because the morphism $\coprod_{i\in I}S_i\to S$ is properly surjective, and $\coprod_{i\in I}S_i\to T$ satisfies $\mathcal{Q}_T$. \qedhere

\end{proof}

\begin{lemma}\label{libresur} Let $p:S':=\mathcal{M(B)}\to S=\mathcal{M(A)}$ be a morphism of $k$-affinoid spaces such that $\mathcal{B}$ satisfies $\mathcal{P_A}$, and let $L$ be a finite free Banach $\mathcal{B}$-module. Then $L$ satisfies the property $\mathcal{P}_\mathcal{A}$.
\end{lemma}

\begin{proof}[Proof] Write $L=\bigoplus_{j\in J} \mathcal{B}$. Let $X$ be a proper $\mathcal{A}$-space provided with a covering by a finite number of affinoid domains and $\mathcal{F}$ a coherent sheaf on $X$. Denote by $(C^*)$ the associated Čech complex. Then by proposition 2.1.7.6 of \autocite{bosch1984non}, we have $C^*\hat{\otimes}_\mathcal{A} L=\bigoplus_{j\in J} (C^*\hat{\otimes}_\mathcal{A}\mathcal{B})$ and because the sum is finite for all $i\geq 0$ we have the equality $ H^i(C^*\hat{\otimes}_\mathcal{A} L)=\bigoplus_{j\in J} H^i(C^*\hat{\otimes}_\mathcal{A} \mathcal{B})$. Because the usual tensor product commute with direct sum, and $\mathcal{B}$ satisfies $\mathcal{P_A}$ we see that $L$ satisfies $\mathcal{P_A}$. \qedhere

\end{proof}

\begin{lemma}\label{quasilisse} Let $p:Y:=\mathcal{M(B)}\to X:=\mathcal{M(A)}$ be a quasi-smooth morphism of $k$-affinoid algebras. Then $\mathcal{B}$ satisfies the property $\mathcal{P_A}$.

\end{lemma}

\begin{proof}[Proof] Let $y\in \mathcal{M(B)}$ and denote by $d$ the relative dimension of $p$ at $y$. Then by corollary 5.3.7 of \autocite{ducros2017families}, there exists an affinoid neighborhood $V_y \subset \mathcal{M(B)}$ such that $p_{\vert V_y}$ is quasi-smooth of relative dimension $d$. By part 1 of $\ref{redaf}$, we can replace $\mathcal{M(B)}$ by $V_y$ and $f$ by $f_{\vert V_y}$ i.e. we can assume that $f$ is quasi-smooth of constant relative dimension $d$ by shrinking again $Y$ if needed.

By corollary 5.3.2 of \autocite{ducros2017families}, we can shrink $Y$ to assume that the coherent $\mathcal{B}$-module $\Omega_{Y/X}$ is free of rank $d$. Now let $f_1,...,f_d$ be analytic functions on $\mathcal{M(B)}$ such that $((df_j)(y))_i$ is a basis of $(\Omega_{\mathcal{B}/\mathcal{A}})_{\mathcal{H}(y)}$. Then by lemma 5.4.5 of \autocite{ducros2017families}, the map $\varphi:Y\to \mathbb{A}^d_X$ is quasi-étale at $y$. Now, the quasi-étale (quasi-smooth of relative dimension zero) locus of $\varphi$ on $Y$ is an open subset (it is even Zariski-open) of $Y$ by theorem 10.7.2 of \autocite{ducros2017families}. 
We can assume that there exist an affinoid neighborhood $U$ of $y$ in $Y$ such that the restriction $\varphi_{\vert U}$ remain quasi-étale, and by an application of part 3 of the lemma $\ref{redaf}$, up to consider a $G$-covering of $Y$, we can assume that $\varphi:Y\to \mathbb{A}^d_X$ is quasi-étale at each point.

By definition of $\mathbb{A}^d_X$ and compactness of $Y$, there exist a compact polydisk $D$ over $\mathcal{A}$ such that $\varphi(Y)\subset D$. Then the induced morphism $Y\to D$ is quasi-étale, and by part 2 of lemma $\ref{redaf}$, it is sufficient to show that $D\to X$ (resp. $Y\to D$)  satisfies $\mathcal{P_\mathcal{A}}$ (resp. $\mathcal{P}_{\mathcal{O}(D)})$. The morphism $D\to X$ obviously satisfies $\mathcal{P_A}$ since if $D=\mathcal{M}(\mathcal{A}\{\underline{R}^{-1}\underline{T}\})$ with $\underline{R}\in (\mathbb{R}_+^*)^d$ a polyradius, $Z$ is a proper $\mathcal{A}$-space provided with a covering by a finite number of affinoid domains, $\mathcal{F}$ is a coherent sheaf on $Z$ and $(C^*)$ is the associated Čech complex then for all $i\geq 0$ we have $H^i(C^*\{\underline{R}^{-1}\underline{T}\})=H^i(C^*)\otimes_\mathcal{A}\mathcal{A}\{\underline{R}^{-1} \underline{T}\}$ because for every exact admissible sequence of Banach $\mathcal{A}$-modules $0\to M'\to M \to M''\to 0$, the sequence $0\to M'\hat{\otimes}_\mathcal{A}\mathcal{A}\{\underline{R}^{-1} \underline{T}\}\to M\hat{\otimes}_\mathcal{A}\mathcal{A}\{\underline{R}^{-1} \underline{T}\}\to M''\hat{\otimes}_\mathcal{A}\mathcal{A}\{\underline{R}^{-1} \underline{T}\}\to 0$ is also exact admissible.

It remains to show that if $q:Y\to D$ is a quasi-étale morphism between affinoid spaces then $\mathcal{O}(Y)$ satisfies $\mathcal{P}_\mathcal{D}$. By part 3 of lemma $\ref{redaf}$, we can argue $G$-locally on $Y$, and assume that $Y$ is an affinoid domain of an étale space $T$ over $D$. Let $p:T\to D$ be the structural morphism and $j:Y\to T$ the inclusion that identifie $Y$ with an affinoid domain of $T$. To show that $\mathcal{O}(Y)$ satisfies $\mathcal{P}_\mathcal{D}$, it is sufficient to show that $q$ satisfies $\mathcal{Q}_D$. Let now suppose that étale morphisms satisfy the property $\mathcal{Q}$. Then $q$ also satisfies $\mathcal{Q}$ by part 2 of lemma $\ref{redQ}$ because affinoid domain embeddings satisfy property $\mathcal{Q}$.


We can now assume that $q_1:Y\to D$ is an étale morphism between non necessary affinoid spaces, and it is sufficient to show that $q_1$ satisfies the property $\mathcal{Q}$. By definition, there exists $Y=\bigcup_{j\in J} Y_j$ a $G$-covering of $Y$ by affinoid domains and $D_j$ affinoid domains of $D$ such that the restriction $q_{1\vert Y_j}:Y_j\to D_j$ is finite and étale. Let $q_2:\coprod_{j\in J} Y_j\to \coprod_{j\in J} D_j$ be the induced finite morphism and $r_2:\coprod_{j\in J} Y_j\to Y$ and $r_1:\coprod_{j\in J} D_j\to D$ the induced $G$-covering. We now have the following commutative square : 

$$ \xymatrix{
    \coprod_{j\in J} Y_j \ar[r]^{q_2} \ar[d]^{r_2}  & \coprod_{j\in J} D_j \ar[d]^{r_1} \\
    Y \ar[r]^{q_1} & D
  }$$
  
By lemma part 3 of $\ref{redQ}$, $q_1$ satisfies the property $\mathcal{Q}_D$ if the morphism $\coprod_{j\in J} Y_j\to D$ satisfies property $\mathcal{Q}_D$. Now, because the morphism $\coprod_{j\in J} D_j\to D$ satisfies property $\mathcal{Q}_D$, again by part 2 of lemma $\ref{redQ}$, it is sufficient to show that the morphism $\coprod_{j\in J} Y_j\to \coprod D_j$ satisfies the property $\mathcal{Q}$.

We can now assume that the morphism $q:Y\to D$ is a finite étale morphism of affinoid spaces and we want to show that it possess the property $\mathcal{Q}$. Let $f:Z\to D$ be a proper ${D}$-space provided with a covering by a finite number of affinoid domains $Z=\bigcup Z_r$. Let $\mathcal{F}$ be a coherent sheaf on $Z$. Denote by $(C^*)$ the associated Čech complex. Let $f'=f\times Y$ and $q'=q\times X$. Then, by properness, we have $H^i(Z\times_D Y,q'^*\mathcal{F})=H^i(C^*\hat{\otimes}_{\mathcal{O}(D)} \mathcal{O}(Y))$ because the inverse image $(q'^{-1}(Z_r))$ induce an affinoid covering of $X\times_D Y$, and because the arrow $\mathcal{O}(D)\to \mathcal{O}(Y)$ is finite, we have $C^i\hat{\otimes}_{\mathcal{O}(D)}\mathcal{O}(Y)=C^i{\otimes}_{\mathcal{O}(D)}\mathcal{O}(Y)$ and because the same arrow is flat by proposition 4.3.1 of \autocite{ducros2017families}, the cohomology of the complex $(C^*)$ commute with (ordinary) tensor product, so we have $H^i(C^*{\otimes}_{\mathcal{O}(D)}\mathcal{O}(Y))=H^i(C^*)\otimes_{\mathcal{O}(D)} \mathcal{O}(Y)$, so  we eventually have the equality $H^i(X\times_D Y,q'^*\mathcal{F})=H^i(C^*)\otimes_{\mathcal{O}(D)} \mathcal{O}(Y)$ and $q$ satisfies property $\mathcal{Q}_D$. \qedhere


\end{proof}

\begin{remaa} \label{freeoverqs}
Combining lemma $\ref{libresur}$ and $\ref{quasilisse}$, we see that if $\mathcal{A}$ is a $k$-affinoid space, and $M$ is any finite free module on a quasi-smooth algebra, then $M$ satisfies the property $\mathcal{P}_\mathcal{A}$.
\end{remaa}

Now, we give a last lemma that will allow us to show the theorem $\ref{aff}$ by induction on the relative dimension of the morphism.

\begin{lemma} \label{induc} Let $p:R\to S$ be a morphism of affinoid spaces with $\mathcal{M(A)}=S$ and $\mathcal{F}$ be a coherent sheaf on $R$ that is flat over $S$. Assume that there exists a quasi-smooth affinoid $S$-space $T=\mathcal{M(C)}$, and $\mathcal{L}$, $\mathcal{N}$ some coherent sheaves on $T$ such that $L:=\mathcal{L}(T)$ is free over $\mathcal{C}$, $\mathcal{N}$ is flat over $S$ and such that there exists a finite $S$-map $\pi:\mathrm{supp}~\mathcal{F}\to T$ and an exact sequence of finite $C$-modules $0\to L\to  H^0(R,\mathcal{F})\to N \to 0$, where $N$ is equal to $\mathcal{N}(T)$.

Assume that $N$ satisfies property $\mathcal{P}_{\mathcal{A}}$. Then $F:=H^0(R,\mathcal{F})$ also satisfies $\mathcal{P}_{\mathcal{A}}$.

\end{lemma}

\begin{proof}[Proof] Let $X$ be a proper $\mathcal{A}$-space provided with a covering by a finite number of affinoid domains $X=\bigcup X_i$ and a coherent sheaf $\mathcal{G}$ on it. Denote by $(C^*)$ the Čech complex associated to $\mathcal{G}$. Let $q_i:Y_i:=\coprod_{i_0<..<i_i} X_{i_0}\cap ... \cap X_{i_i}\to X$ be the morphism of $k$-affinoid spaces given by inclusion. Then we have the equality $H^0(Y_i,q_i^*\mathcal{G})=C^i$.
We also have an exact sequence of coherent $\mathcal{O}_T$-modules $0\to \mathcal{L}\to \pi_* \mathcal{F}\to \mathcal{N}\to 0$. We can summarize the situation by the following commutative cartesian square, where $Z:= T\times_S Y_i$:

$$ \xymatrix{
     Z\ar[r] \ar[d]  & Y_i  \ar[d] \\
    T \ar[r] & S
  }$$
By flatness of $\mathcal{N}$, we can apply theorem 4.5.7 of \autocite{ducros2017families} to this square, the previous exact sequence and the coherent sheaf $q_i^*\mathcal{G}$ on $Y_i$ and we have the following exact sequence for all $i\geq 0$ : $0\to C^i\hat{\otimes}_\mathcal{A} L\to C^i\hat{\otimes}_\mathcal{A} F\to C^i\hat{\otimes}_\mathcal{A} N\to 0$. 

This short exact sequence of complexes of modules now induces a long exact sequence of modules, and this shows that the second row of the following commutative diagram is exact at the middle:

$$ \xymatrix{
    0 \ar[r] \ar[d]  & H^n(C^*)\otimes_\mathcal{A} L \ar[d] \ar[d] \ar[r] \ar[d] & H^n(C^*)\otimes_\mathcal{A} F  \ar[r] \ar[d] & H^n(C^*)\otimes_\mathcal{A} N \ar[r]\ar[d] & 0\ar[d] \\
    0 \ar[r] & H^n(C^*\hat{\otimes}_\mathcal{A}L) \ar[r] & H^n(C^*\hat{\otimes}_\mathcal{A} F) \ar[r] & H^n(C^*\hat{\otimes}_\mathcal{A} N)\ar[r] & 0
  }$$

The first row of this commutative diagram is exact at each of its term because since $\mathcal{N}$ is flat over $S$, the $\mathcal{A}$-module $N$ is $\mathcal{A}$ flat, so we have for all $n\geq 0$ the equality $\mathrm{Tor}^\mathcal{A}_1~(C^n,N)=0$.
Now, by diagram-chasing, since $N$ satisfies $\mathcal{P_A}$, the arrow $H^n(C^*\hat{\otimes}_\mathcal{A} F)\to H^n(C^*\hat{\otimes}_\mathcal{A} N)$ is surjective, and this shows that the connecting morphism $H^n(C^*\hat{\otimes}_\mathcal{A} N)\to H^{n+1}(C^*\hat{\otimes}_\mathcal{A} L)$ from the long exact sequence is injective, so the arrow $H^{n+1}(C^*\hat{\otimes}_\mathcal{A} L)\to H^{n+1}(C^*\hat{\otimes}_\mathcal{A} F)$ is injective, so the second row is also exact on its left part for $n\geq 1$ applying the above with $n$ instead of $n+1$. The case $n=0$ is also exact since we have a commutative diagram of modules with exact rows:

$$ \xymatrix{
     0\ar[r]  & C^0\hat{\otimes}_\mathcal{A} L  \ar[d]\ar[r] & C^0\hat{\otimes}_\mathcal{A} F\ar[d] \\
    0 \ar[r] & C^1\hat{\otimes}_\mathcal{A} L \ar[r] & C^1\hat{\otimes}_\mathcal{A} F
  }$$
And by diagram-chasing, the canonical arrow $H^0(C^0\hat{\otimes}_\mathcal{A} L)\to H^0(C^0 \hat{\otimes}_\mathcal{A} F)$ is also injective.
Now, since $N$ satisfies $\mathcal{P}_\mathcal{A}$, and $L$ also by $\ref{freeoverqs}$ because it is free over a quasi-smooth algebra over $\mathcal{A}$, on the three descending morphism of the diagram, two are isomorphisms, and by diagram chase, we dedude that $H^n(C^*)\otimes_\mathcal{A} F\to H^{n}(C^*\hat{\otimes}_\mathcal{A} F)$ is also an isomorphism and it is what we aimed to show.
\end{proof}

We will now use the previous lemma repeatedly to show : 

\begin{lemma} \label{caspur}Let $\mathcal{A}$ be an affinoid algebra, $d\geq 0$ an integer, $T$ a quasi-smooth space $T\to \mathcal{M(A)}$ purely of relative dimension $d$, and $\mathcal{F}$ a coherent sheaf on $T$ that is flat on $\mathcal{A}$. Then for all $i\geq 0$, $H^0(T,\mathcal{F})$ satisfies the property $\mathcal{P}_\mathcal{A}$.
\end{lemma}

\begin{proof}[Proof]
We will show this by induction on the relative dimension $d\geq 0$. 

Assume $d=0$. Then by $\ref{deviss}$, there exist a quasi-smooth space $T_1$ over $\mathcal{A}$, of pure relative dimension zero and a finite morphism of $\mathcal{A}$-analytic spaces $\pi_1:\mathrm{supp}~ \mathcal{F}\to T_1$, with a coherent sheaf $\mathcal{L}_1$ on $T_1$ whose global section $L_1:=H^0(T_1,\mathcal{L}_1)$ are free over $\mathcal{A}$ and such that we have an isomorphism $0\to L_1\to H^0(T,\mathcal{F})\to 0$. From this equality, we deduce that $H^0(T,\mathcal{F})$ is free over a quasi-smooth algebra over $\mathcal{A}$, so by remark $\ref{freeoverqs}$, the $\mathcal{A}$-module $H^0(T,\mathcal{F})$ satisfies the property $\mathcal{P_A}$.

Assume now that the proposition hold for all $d'\leq d$, and let $T\to \mathcal{M(A)}$ be a quasi-smooth space purely of dimension $d+1$ with $\mathcal{F}$ a coherent sheaf on $T$. Then using the same notation as proposition $\ref{deviss}$, we have an affinoid quasi-smooth $\mathcal{A}$-space $T_1=\mathcal{M(C_1)}$ of pure relative dimension $n_r\leq d+1$, an affinoid quasi-smooth $\mathcal{A}$-space $T_2=\mathcal{M(C_2)}$ of pure relative dimension $n_r\leq d$, some coherent $\mathcal{C}_i$-module $L_i$ and $P_i$ for $i\in \{1;2\}$ on $\mathcal{C}_i$ such that $L_i$ is free over $\mathcal{C}_i$, $P_i$ is (analytically) flat over $\mathcal{M(A)}$ and some finite morphism $\pi_1:\mathrm{supp}~ \mathcal{F}\to T_1$ and $\pi_2:\mathrm{supp} ~ P_1\to T_2$. Now, we have an exact sequence $0\to L_2\to H^0(T_1,P_1)\to P_2\to 0$, and by induction hypothesis, $P_2$ satisfies $\mathcal{P_A}$, and $L_2$ is free over $C_2$, so by the previous lemma $\ref{induc}$, $P_1$ satisfies $\mathcal{P_A}$. Now applying again the lemma $\ref{induc}$ to the exact sequence of $\mathcal{C}_1$ modules $0\to L_1\to H^0(T,\mathcal{F})\to P_1\to 0$ we get that $H^0(T,\mathcal{F})$ satisfies $\mathcal{P}_\mathcal{A}$.
\end{proof}

Now, theorem $\ref{aff}$ is just an easy consequence of the previous lemmas.
\begin{proof}[Proof]
Let $p:\mathcal{M(B)}\to \mathcal{M(A)}$ be a flat morphism between affinoid spaces. We want to show that the $\mathcal{A}$-module $\mathcal{B}$ satisfies the property $\mathcal{P}_\mathcal{A}$. By lemma $\ref{redaf}$, the property is $G$-local on $\mathcal{M(B)}$ and it is sufficient to show it on an affinoid neighborhood of any point.

So let $y\in \mathcal{M(B)}$ and let $n$ be the relative dimension of $p$ at $y$. Using notation of the proposition $\ref{deviss}$, there exist an affinoid neighborhood $V=\mathcal{M}(\mathcal{B}_V)$ of $y$ in $Y$ and a quasi-smooth $\mathcal{A}$-space $T_1=\mathcal{M(C_1)}$ purely of dimension $d$ smaller than $n$, finite $C_1$-modules $P_1$ and $L_1$ such that $L_1$ is free over $\mathcal{C}_1$ and a finite morphism $\pi_1:V\to T_1$ such that there exist an exact sequence of $\mathcal{C}_1$-modules $0\to L_1\to \mathcal{B}_V\to P_1\to 0$.

Now by the $\ref{caspur}$, $P_1$ satisfies $\mathcal{P_A}$ and $L_1$ is free over a quasi-smooth $\mathcal{A}$-algebra, so by lemma $\ref{induc}$, the $\mathcal{A}$-module $\mathcal{B}_V$ satisfies $\mathcal{P_A}$, this shows that the property holds for $\mathcal{B}$ and the theorem $\ref{aff}$ is now proved. \qedhere

\end{proof}

\section{Proper base change}

We now want to show a proper base change theorem for coherent modules. We remind two following theorems that are stated at section 5 of the chapter 2 of \autocite{mumford1974abelian} for the cohomology of proper schemes.

\begin{prop} \label{Homoex} Let $A$ be a noetherian ring and $C^*$ be a complex of $A$-modules such that its cohomology groups $H^i(C^*)$ are finitely generated $A$-modules and $C^p\neq \{0\}$ if and only if $0\leq p\leq n$. Then there exist a complex $K^*$ of finitely generated $A$-modules such that $K^p \neq \{0\}$ if and only if $0\leq p \leq n$ and $K^p$ is free for $1\leq p\leq n$, and a quasi-isomorphism of complexes of $A$-modules $\varphi:K^*\to C^*$. Moreover, if the $C^p$ are $A$-flat, then $K^0$ can be taken to be $A$-flat.

\end{prop}

\begin{prop} \label{homoprop} Let $A$ be a noetherian ring,  and $C^*$, $K^*$ be any finite complexes of flat $A$-modules and let $\varphi:K^*\to C^*$ be a quasi-isomorphism of complexes of $A$-modules. Then for every $A$-algebra $B$, the maps $H^p(K^*\otimes_A B)\to H^p(C^* \otimes_A B)$ are isomorphisms for all $p\in \mathbb{Z}$ i.e. the natural morphism $\varphi\otimes_A B$ is a quasi-isomorphism.
\end{prop}

If we have a quasi-coherent sheaf $\mathcal{F}$ on a proper space $X$ over a noetherian affine scheme $A$, we can apply these two proposition to the Čech complex of $\mathcal{F}$, and we have the existence of a complex $K^*$ of $A$-module that computes the cohomology of the space $X$ universally, i.e. such that we have for every $A$-algebra $B$ and every integer $n$ the equality $H^n(X_B,\mathcal{F}_B)=H^n(K^*_A \otimes B)$.

\begin{defia} Let $\mathcal{A}$ be an affinoid algebra, and $X$ a proper $\mathcal{A}$-analytic space. Let $\mathcal{F}$ be a coherent sheaf on $X$ that is flat over $\mathcal{A}$, and let $X=\bigcup_{i\in I} X_i$ be $G$-covering of $X$ by a finite number of affinoid domains. Denote by $C^*$ the Čech complex associated to $\mathcal{F}$ relatively to this covering. Given these data, we will say that a morphism of affinoid algebras $p:\mathcal{M(B)}\to \mathcal{M(A)}$ satisfies the property ${{\mathcal{R}_\mathcal{A}}}$ if for every finite complex of finitely generated flat $A$-module $K^*$ and every quasi-isomorphism $\varphi:K^*\to C^*$ of complexes of $\mathcal{A}$-modules then $\varphi$ induce an isomorphism $H^n(K^*\hat{\otimes}_\mathcal{A} \mathcal{B})\to H^n(C^*\hat{\otimes}_\mathcal{A} \mathcal{B})$ i.e. $\varphi\hat{\otimes}_\mathcal{A} \mathcal{B}$ is a quasi-isomorphism.
\end{defia}

From now on, all the data involving the previous definition will be fixed.

\begin{remaa} \label{finiteR} Let $p:\mathcal{M(B)}\to \mathcal{M(A)}$ be a finite morphism between $k$-affinoid algebras. Then $p$ satisfies the property ${\mathcal{R}_\mathcal{A}}$. In fact, the complex $C^*$ satisfies the hypothesis of the proposition $\ref{homoprop}$, so for every finite flat complex of finitely generated $A$-module $K^*$ and every quasi-isomorphism $\varphi:K^*\to C^*$ of complexes of $\mathcal{A}$-modules, we have an isomorphism $H^n(K^*{\otimes}_\mathcal{A} \mathcal{B})\to H^n(C^*{\otimes}_\mathcal{A} \mathcal{B})$ and since $C^i$ and $K^i$ are noetherian and $p$ is finite, we have the equality $K^*\hat{\otimes}_\mathcal{A} \mathcal{B}=K^*{\otimes}_\mathcal{A} \mathcal{B}$ and $C^*\hat{\otimes}_\mathcal{A} \mathcal{B}=C^*{\otimes}_\mathcal{A} \mathcal{B}$, so $\varphi\hat{\otimes} \mathcal{B}$ is a quasi-isomorphism.

\end{remaa}

\begin{remaa}\label{flatproper} Let $p:\mathcal{M(B)}\to \mathcal{M(A)}$ be a flat morphism between affinoid algebras. Then $p$ satisfies ${\mathcal{R}_\mathcal{A}}$. In fact, let $K^*$ be a finite complex of finitely generated flat $A$-modules and $\varphi$ a quasi-isomorphism $\varphi:K^*\to C^*$ of complexes of $\mathcal{A}$-modules. Because $K^i$ is finite for all $i\in \mathbb{Z}$, and the $\mathcal{A}$-algebra $\mathcal{B}$ is flat, using theorem $\ref{aff}$, the arrow $H^n(K^* \hat{\otimes}_\mathcal{A}\mathcal{B})\to H^n(C^* \hat{\otimes}_\mathcal{A} \mathcal{B})$ is now identified with the arrow $H^n(K^*)\otimes_\mathcal{A}\mathcal{B}\to H^n(C^*)\otimes_\mathcal{A}\mathcal{B}$ and the latter is an isomorphism by definition of the complex $K^*$.

\end{remaa}

\begin{lemma} \label{fieldle} Let $L$ be a non-archimedean field extension of $k$ and $\mathcal{A}$ be a $k$-affinoid algebra. Then the natural morphism of analytic spaces $p:\mathcal{M}({A}_L)\to \mathcal{M(A)}$ satifies $\mathcal{R}_\mathcal{A}$.

\end{lemma}

\begin{proof}[Proof] Let $K^*$ be a finite complex of finitely generated flat $A$-module and $\varphi$ a quasi-isomorphism $\varphi:K^*\to C^*$ of complexes of $\mathcal{A}$-modules. Then the differential of the Čech complex $C^*$ are admissible, and by theorem 1 of part 3 of \autocite{gruson1966theorie}, the functor $\hat{\otimes}_k L$ from $k$-Banach modules to $L$-Banach modules transform admissible exact sequences into admissible exact sequences, so for all $n\in\mathbb{Z}$, the morphism $H^n(K^*\hat{\otimes}_k L)\to H^n(C^*\hat{\otimes}_k L)$ is identified with the morphism $H^n(K^*)\hat{\otimes}_k L\to H^n(C^*)\hat{\otimes}_k L$ and the latter is an isomorphism by hypothesis, so $p$ satisfies $\mathcal{R}_\mathcal{A}$. \qedhere

\end{proof}

\begin{remaa}Note the previons lemma and the remark just before did not use the flatness of $K^*$.

\end{remaa}

\begin{lemma} \label{Rlemme}
Let $q:\mathcal{M(D)}\to \mathcal{M(B)}$ and $p:\mathcal{M(B)}\to \mathcal{M(A)}$ be morphisms of $k$-affinoid spaces. 
\begin{enumerate}
\item Assume that $p$ satisfies ${\mathcal{R}_\mathcal{A}}$ (e.g. flat) and $q$ is finite. Then $p\circ q $ satisfies ${\mathcal{R}_\mathcal{A}}$.
\item Assume that there exist a $G$-covering $\mathcal{M(B)}=\bigcup_{j\in J} \mathcal{M}(\mathcal{B}_j)$ of $\mathcal{M(B)}$ by a finite number of affinoid domains such that for all $j\in J$, the induced arrow $\mathcal{M}(\mathcal{B}_j)\to \mathcal{M(A)}$ satisfies ${\mathcal{R}_\mathcal{A}}$. Then $p$ satisfies ${\mathcal{R}_\mathcal{A}}$.
\item Assume that $p$ satisfies ${\mathcal{R}_\mathcal{A}}$ and $q$ satisfies $\mathcal{R}_\mathcal{B}$. Then $p\circ q $ satisfies $\mathcal{R}_A$.
\end{enumerate}
\end{lemma}

\begin{proof}[Proof]
Let $K^*$ be a finite complex of finitely generated flat $A$-modules and $\varphi$ a quasi-isomorphism $\varphi:K^*\to C^*$ of complexes of $\mathcal{A}$-modules.

For the first point, we have a homomorphism $\varphi_\mathcal{B}:K \hat{\otimes}_\mathcal{A} \mathcal{B}\to C^* \hat{\otimes}_\mathcal{A} \mathcal{B}$ and by flatness of $p$ and remark $\ref{flatproper}$, this arrow induces an isomorphism $H^n(K^*\hat{\otimes}_\mathcal{A}\mathcal{B})\to H^n(C^*\hat{\otimes}_\mathcal{A}\mathcal{B})$ for all $n\in \mathbb{Z}$. Now we can apply $\ref{homoprop}$ to the complex $K^*\hat{\otimes}_\mathcal{A}\mathcal{B}$ and $C^*\hat{\otimes}_\mathcal{A}\mathcal{B}$ over $\mathcal{B}$, to obtain an isomorphism $H^n((K^*\hat{\otimes}_\mathcal{A}\mathcal{B})\otimes_\mathcal{B} \mathcal{D})\to H^n((C^*\hat{\otimes}_\mathcal{A}\mathcal{B})\otimes_\mathcal{B} \mathcal{D})$, and since $K^i$ is a finite $\mathcal{A}$-module for all $i\in \mathbb{Z}$ and $C^*\hat{\otimes}_\mathcal{A}\mathcal{B}$ is noetherian and $q$ is finite, we obtain an isomorphism $H^n((K^*\hat{\otimes}_\mathcal{A}\mathcal{B})\hat{\otimes}_\mathcal{B} \mathcal{D})\to H^n((C^*\hat{\otimes}_\mathcal{A}\mathcal{B})\hat{\otimes}_\mathcal{B} \mathcal{D})$, and this shows that $p\circ q $ satisfies ${\mathcal{R}_\mathcal{A}}$.

For the second point, we have an isomorphism $H^n(K^*\hat{\otimes}_\mathcal{A} {\bigoplus_{j\in J}\mathcal{B}_j})\to H^n(C^*\hat{\otimes}_\mathcal{A} {\bigoplus_{j\in J}\mathcal{B}_j})$. Since cohomology and completed tensor product commute with finite direct sums and the affinoid domain inclusion $\mathcal{M}(\mathcal{B}_i)\to \mathcal{M(B)}$ is flat, we can use $\ref{aff}$ to get an isomorphism $H^n(K^*)\hat{\otimes}_\mathcal{A} \bigoplus_{j\in J}\mathcal{B}_i\to H^n(C^*)\hat{\otimes}_\mathcal{A}\bigoplus_{j\in J}\mathcal{B}_i$. By flatness of the morphism $\mathcal{M}(\bigoplus_{j\in J} \mathcal{B}_j)\to\mathcal{M(B)}$, and descent proposition 3.11 of \autocite{daylies2021descente}, we deduce that the map $H^n(K^*\hat{\otimes}_\mathcal{A}\mathcal{B})\to H^n(C^*\hat{\otimes}_\mathcal{A}\mathcal{B})$ is an isomorphism, and $p$ satisfies ${\mathcal{R}_\mathcal{A}}$.

For the last point, since $p$ satisfies $\mathcal{R}_\mathcal{A}$, we have an isomorphism $H^n(K^*\hat{\otimes}_\mathcal{A} \mathcal{B})\to H^n(C^*\hat{\otimes}_\mathcal{A} \mathcal{B})$, and this show that the morphism of complexes of $\mathcal{B}$-modules $\varphi_\mathcal{B}$ induce an isomorphism on cohomology, so since $q$ satisfies $\mathcal{R}_\mathcal{B}$, we have an isomorphism $H^n(K^*\hat{\otimes}_\mathcal{A} \mathcal{B} \hat{\otimes}_\mathcal{B} \mathcal{D} )\to H^n(C^*\hat{\otimes}_\mathcal{A} \mathcal{B} \hat{\otimes}_\mathcal{B} \mathcal{D})$ so $p\circ q $ satisfies $\mathcal{R}_\mathcal{A}$.
\end{proof}

\begin{theoa} \label{propBC} Let $L$ be a non-archimedean field extension of $k$, and let $\mathcal{B}$ be an $L$-affinoid algebra. Then for every $k$-affinoid algebra and every morphism $p:\mathcal{M(B)}\to \mathcal{M(A)}$ of analytic spaces, $p$ satisfies the property ${\mathcal{R}_\mathcal{A}}$. In particular, for every proper $\mathcal{A}$-space $X$, and every coherent sheaf $\mathcal{F}$ on it that is flat over $\mathcal{M(A)}$, there exist a finite complex of finite and projective $\mathcal{A}$-modules $K^*$ such that for every $\mathcal{A}$-algebra $\mathcal{B}$ as above, we have an isomorphism $H^n(K^*\otimes_\mathcal{A} \mathcal{B})\to H^n(X_\mathcal{B},\mathcal{F}_\mathcal{B})$.
\end{theoa}

\begin{proof}[Proof]
The morphism $p$ can be written as the composition $\mathcal{M(B)}\to\mathcal{M}(\mathcal{A}_L)\to \mathcal{M(A)}$, and $\mathcal{M}(\mathcal{A}_L)\to \mathcal{M(A)}$ satisfies $\mathcal{R}_\mathcal{A}$ by \ref{fieldle} so using $\ref{Rlemme}$, we can assume that $L=k$, so we can assume that $\mathcal{B}$ is $k$-affinoid.

Let $p:\mathcal{M(B)}\to \mathcal{M(A)}$ be a morphism of $k$-affinoid spaces, and let $y\in \mathcal{B}$. By corollary 4.7 of \autocite{ducros_2007}, there exist an affinoid neighboorhood $V_y$ of $y$ in $\mathcal{M(B)}$, a relatively smooth $\mathcal{A}$-space $T_y$, an affinoid domain $W_y$ of $T_y$ and a finite morphism $V_y\to W_y$ such that the restriction $p_{\vert V_y}$ can be written as the composition $V_y\to W_y\to T_y\to \mathcal{M(A)}$.

We can extract a finite $G$-covering $\mathcal{M(B)}=\bigcup_{i\in I}V_i$ of $\mathcal{M(B)}$ from the $G$-covering $\mathcal{M(B)}=\bigcup_{y\in \mathcal{M(B)}}V_y$ and by part 2 of lemma $\ref{Rlemme}$, it is sufficient to show that each restriction $p_{\vert V_i}\to \mathcal{M(A)}$ satisfies the property $\mathcal{R}_\mathcal{A}$. Since $V_i\to W_i$ is finite, by remark $\ref{finiteR}$, it satisfies $\mathcal{R}_{\mathcal{O}(W_i)}$ and since $W_i\to \mathcal{M(A)}$ is flat, by remark $\ref{flatproper}$, it satisfies $\mathcal{R}_\mathcal{A}$. By the third point of lemma $\ref{Rlemme}$, we have that $V_i\to \mathcal{M(A)}$ satisfies $\mathcal{R}_\mathcal{A}$.
\end{proof}

\begin{remaa}
Let $\mathcal{M(A)}$ be a $k$-affinoid spaces and $s\in \mathcal{M(A)}$. Then we can identifie $\mathcal{M(H}(s))$ with a rigid point of $\mathcal{M(A}_{\mathcal{H}(s)})$, and write the inclusion $\mathcal{M(H}(s))\to \mathcal{M(A)}$ as a closed immersion $\mathcal{M(H}(s))\to \mathcal{M(A}_{\mathcal{H}(s)})$ followed by a base field extension morphism $\mathcal{M(A}_{\mathcal{H}(s)})\to \mathcal{M(A)}$. These two morphism satisfies the property $\mathcal{R}_{\mathcal{A}_{\mathcal{H}(s)}}$ (resp the property $\mathcal{R}_\mathcal{A}$), so by lemma $\ref{Rlemme}$, the composition $\mathcal{M}(\mathcal{H}(s))\to \mathcal{M(A)}$ satisfies the property $\mathcal{R}_\mathcal{A}$.

In particular, for any proper space $X$ over $\mathcal{M(A)}$, any flat coherent sheaf $\mathcal{F}$ over $X$ and any $s\in \mathcal{M(A)}$, there exist a finite complex $K^*$ of finite and projective $\mathcal{A}$-modules $K^*$ such that we have an isomorphism $H^n(K^*\otimes_\mathcal{A} \mathcal{H}(s))\to H^n(X_s,\mathcal{F}_s)$. This remark led us to the following corollary that was stated without proof in corollary 3.3.11 of \autocite{berkovich2012spectral}.

\end{remaa}

\begin{coroa}\label{coreuler} Let $f:X\to S$ be a proper morphism of $k$-analytic spaces, and let $\mathcal{F}$ be a coherent sheaf on $X$ which is flat over $S$. Then:
\begin{enumerate}
    \item for all $p\geq 0$ and $n\geq 0$, the set $\{s\in S\vert \dim_{\mathcal{H}(s)}H^p(X_s,\mathcal{F}_s)\geq n\}$ is a Zariski-closed subset of $S$.
    \item the Euler characteristic $\chi:X\to \mathbb{Z}$ defined by 
    \begin{equation}
        s\mapsto \chi(\mathcal{F}_s):=\sum_{p=0}^\infty (-1)^p \dim_{\mathcal{H}(s)} H^p(X_s,\mathcal{F}_s)
    \end{equation}
    is $G$-locally (and locally) constant on $S$.
\end{enumerate}

\end{coroa}

\begin{proof}[Proof]
Both result are $G$-local on $S$ so we can assume that $S=\mathcal{M(A)}$ is $k$-affinoid and $X$ is a proper $\mathcal{A}$-space. By the previous remark, there exist a finite complex $K^*$ of finite and projective $\mathcal{A}$-modules $K^*$ such that we have an isomorphism $H^n(K^*\otimes_\mathcal{A} \mathcal{H}(s))\to H^n(X_s,\mathcal{F}_s)$. For $p\in \mathbb{Z}$, denote by $\delta^p:K^p\to K^{p+1}$ the $p$-th differential of the complex $K^*$. Then for all $s\in S$ we have the equality:

\begin{equation} \label{equdim}
\begin{split}
   \dim_{\mathcal{H}(s)} H^p(X_s,\mathcal{F}_s) = \dim_{\mathcal{H}(s)} \mathrm{Ker}(\delta^p\hat{\otimes}_\mathcal{A} \mathcal{H}(s))-\dim_{\mathcal{H}(s)} \mathrm{Im} (\delta^{p-1}\hat{\otimes}_\mathcal{A} \mathcal{H}(s))  \\ 
     =\dim_{\mathcal{H}(s)}(K^p\hat{\otimes}_\mathcal{A} \mathcal{H}(s))-\dim_{\mathcal{H}(s)} \mathrm{Im} (\delta^{p}\hat{\otimes}_\mathcal{A} \mathcal{H}(s))-\dim_{\mathcal{H}(s)} \mathrm{Im} (\delta^{p-1}\hat{\otimes}_\mathcal{A} \mathcal{H}(s))
\end{split}
\end{equation}

Now $s\mapsto \dim_{\mathcal{H}(s)} \mathrm{Im} (\delta^{p}\hat{\otimes}_\mathcal{A} \mathcal{H}(s))$ is a lower semicontinous function for the Zariski topology on $S$ because the locus where this dimension is less than a number $N$ is the Zariski closed subset defined by the vanishing of some finite number of minor of the matrix $K^p\to K^{p+1}$, so $s\mapsto \dim_{\mathcal{H}(s)} H^p(X_s,\mathcal{F}_s)$ is upper semicontinuous for the Zariski topology and this shows the first point.

For the second point, using $\ref{equdim}$, we have for all $s\in S$ the equality $$ \chi(\mathcal{F}_s)=\sum_{p=0}^\infty (-1)^p\dim_{\mathcal{H}(s)}(K^p\hat{\otimes}_\mathcal{A} \mathcal{H}(s)) $$ which is $G$-locally constant on $S$ since $K^p$ is flat over $\mathcal{A}$ using lemma 4.1.14 of \autocite{ducros2017families}. \qedhere

\end{proof}

The rest of this section follows the chapter 2, section 5 of the book of \autocite{mumford1974abelian} of Mumford. The proofs are quite similar, and we will give the arguments that differs from scheme theory.

\begin{lemma}
Let $Y$ be a reduced analytic space, $\mathcal{F}$ a coherent sheaf on it such that for all $y\in Y$, we have the equality $\dim_{\mathcal{H}(y)}[\mathcal{F}\otimes_{\mathcal{O}_Y} \mathcal{H}(y)]=r$. Then, $\mathcal{F}$ is a locally free of rank $r$ on $Y$.
\end{lemma}

\begin{proof} The proof is exactly the same proof as in lemma 1 of section 5 of the book \autocite{mumford1974abelian}, combined with 2.5.2 of  \autocite{ducros2017families} that is specific to the analytical case. 

\end{proof}

The next lemma is the analogue of lemma 2 of section 5 of \autocite{mumford1974abelian}.

\begin{lemma}
Let $S=\mathcal{M(A)}$ be a reduced $k$-affinoid space and let $\varphi:\mathcal{F}\to \mathcal{D}$ be a morphism of coherent locally free $\mathcal{O}_S$-sheaves. If $\dim_{\mathcal{H}(s)}[\textrm{Im}(\varphi\otimes \mathcal{H}(s))]$ is locally constant, then there are splittings :

  $$ \mathcal{F}\simeq \mathcal{F}_1 \oplus \mathcal{F}_2$$
  
  $$\mathcal{D}\simeq \mathcal{D}_1\oplus \mathcal{D}_2$$
  
  such that $\varphi_{\vert \mathcal{F}_1}=0$, $\textrm{Im}(\varphi)\subset \mathcal{D}_1$, and $\varphi_{\vert \mathcal{F}_2}\to \mathcal{D}_1$ is an isomorphism.

\end{lemma}

\begin{proof}
The proof is the same as in section 5 of \autocite{mumford1974abelian} once we know that if $\mathcal{F}$ is a coherent locally free $\mathcal{O}_S$-module, the associated $\mathcal{A}$-module $M:=\mathcal{H}^0(S,\mathcal{F})$ is projective. 

So let $\mathcal{F}$ be a coherent locally free $\mathcal{O}_S$-module, $S^{\textrm{alg}}:=\Spec \mathcal{A}$ and let $s\in S$ be a rigid point. It is sufficient to show that if $x$ is the image of $s$ in $S^\textrm{alg}$, then $M_x$ is free. We know that $M\otimes_\mathcal{A}\mathcal{O}_{S,s}$ is a free $\mathcal{O}_{S,s}$-module, so by faithfully flatness of $\mathcal{O}_{S,s}$ over $\mathcal{O}_{S^{\textrm{alg}},x}$, we knows that $M\otimes_\mathcal{A}\mathcal{O}_{S^{\textrm{alg}},x}$ is projective by \autocite{https://doi.org/10.48550/arxiv.1011.0038} so it is free over $\mathcal{O}_{S^{\textrm{alg}},x}$, and $M$ is a projective $\mathcal{A}$-module.
\end{proof}

From these two lemma, we deduce the following theorem, whose proof is the same as in \autocite{mumford1974abelian}, because since the complex $K^*$ and the cohomology groups are finite modules, there is no completion in the tensors products. 

\begin{theoa}\label{finver}

Let $f:X\to S$ be a proper morphism of $k$-analytic spaces with $S$ connected and reduced, and $\mathcal{F}$ a coherent sheaf on $X$ that is flat on $S$. Then, for all $p\in \mathbb{N}$, there is an equivalence between:
\begin{enumerate}
    \item the function $s\in S\mapsto \dim_{\mathcal{H}(s)} H^p(X_s,\mathcal{F}_s)$ is constant ;
    \item the sheaf $R^p f_*(\mathcal{F})$ is locally free on $S$, and for all $s\in S$, the natural map $R^p f_*(\mathcal{F})\otimes_{\mathcal{O}_S} \mathcal{H}(s)\to H^p(X_s,\mathcal{F}_s)$ is an isomorphism.
\end{enumerate}
If these equivalent conditions are fulfilled, then the natural map $R^{p-1} f_*(\mathcal{F})\otimes_{\mathcal{O}_S} \mathcal{H}(s)\to H^{p-1}(X_s,\mathcal{F}_s)$ is an isomorphism for all $s\in S$.

\end{theoa}

\printbibliography

@misc{https://doi.org/10.48550/arxiv.1101.0683,
  doi = {10.48550/ARXIV.1101.0683},
  
  url = {https://arxiv.org/abs/1101.0683},
  
  author = {Ducros, Antoine},
  
  title = {Étale cohomology of schemes and analytic spaces},
  
  publisher = {arXiv},
  
  year = {2011},
  
  copyright = {arXiv.org perpetual, non-exclusive license}
}

@misc{https://doi.org/10.48550/arxiv.1011.0038,
  doi = {10.48550/ARXIV.1011.0038},
  
  url = {https://arxiv.org/abs/1011.0038},
  
  author = {Perry, Alexander},
  
  title = {Faithfully flat descent for projectivity of modules},
  
  publisher = {arXiv},
  
  year = {2010},
  
  copyright = {arXiv.org perpetual, non-exclusive license}
}

@article{gruson1966theorie,
  title={Th{\'e}orie de Fredholm $ p $-adique},
  author={Gruson, Laurent},
  journal={Bulletin de la Soci{\'e}t{\'e} Math{\'e}matique de France},
  volume={94},
  pages={67--95},
  year={1966}
}

@book{mumford1974abelian,
  title={Abelian varieties},
  author={Mumford, David and Ramanujam, Chidambaran Padmanabhan and Manin, Jurij Ivanovi{\v{c}}},
  volume={3},
  year={1974},
  publisher={Oxford university press Oxford}
}

@article{daylies2021descente,
  title={Descente fid$\backslash$element plate et alg$\backslash$'ebrisation en g$\backslash$'eom$\backslash$'etrie de Berkovich},
  author={Daylies, Mathieu},
  journal={arXiv preprint arXiv:2103.10490},
  year={2021}
}

@article{kiehl1967endlichkeitssatz,
  title={Der Endlichkeitssatz f{\"u}r eigentliche Abbildungen in der nichtarchimedischen Funktionentheorie},
  author={Kiehl, Reinhardt},
  journal={Inventiones mathematicae},
  volume={2},
  number={3},
  pages={191--214},
  year={1967},
  publisher={Springer}
}

@misc{stacks-project,
  author       = {The {Stacks project authors}},
  title        = {The Stacks project},
  howpublished = {\url{https://stacks.math.columbia.edu}},
  year         = {2021},
}

@misc{bosch1984non,
  title={Non-Archimedean analysis, volume 261 of Grundlehren der Mathematischen Wissenschaften [Fundamental Principles of Mathematical Sciences]},
  author={Bosch, Siegfried and G{\"u}ntzer, Ulrich and Remmert, Reinhold},
  year={1984},
  publisher={Springer-Verlag, Berlin}
}

@book{berkovich2012spectral,
  title={Spectral theory and analytic geometry over non-Archimedean fields},
  author={Berkovich, Vladimir G},
  number={33},
  year={2012},
  publisher={American Mathematical Soc.}
}

@article{PMIHES_1993__78__5_0,
     author = {Berkovich, Vladimir G.},
     title = {\'Etale cohomology for non-Archimedean analytic spaces},
     journal = {Publications Math\'ematiques de l'IH\'ES},
     pages = {5--161},
     publisher = {Institut des Hautes \'Etudes Scientifiques},
     volume = {78},
     year = {1993},
     zbl = {0804.32019},
     mrnumber = {95c:14017},
     language = {en},
     
}

@article{ducros_2007, title={Variation de la dimension relative en géométrie analytique p-adique}, volume={143}, DOI={10.1112/S0010437X07003193}, number={6}, journal={Compositio Mathematica}, publisher={London Mathematical Society}, author={Ducros, Antoine}, year={2007}, pages={1511–1532}}

@article {ducros2017families,
    AUTHOR = {Ducros, Antoine},
     TITLE = {Families of {B}erkovich spaces},
   JOURNAL = {Ast\'{e}risque},
  FJOURNAL = {Ast\'{e}risque},
      YEAR = {2018},
    NUMBER = {400},
     PAGES = {vii+262},
      ISSN = {0303-1179},
      ISBN = {978-2-85629-885-5},
   MRCLASS = {14G22},
  MRNUMBER = {3826929},
MRREVIEWER = {Marco Maculan},
}

\end{document}